\documentclass[11pt,reqno]{amsart}

\usepackage[T1]{fontenc}
\usepackage{newtxtext,newtxmath}
\usepackage{microtype}
\usepackage{mathtools}
\usepackage{enumitem}
\usepackage{array}
\usepackage[hidelinks]{hyperref}
\usepackage[margin=1.15in]{geometry}

\newtheorem{theorem}{Theorem}[section]

\newtheorem{conjecture}[theorem]{Conjecture}
\newtheorem{proposition}[theorem]{Proposition}
\newtheorem{lemma}[theorem]{Lemma}
\theoremstyle{definition}
\newtheorem{definition}[theorem]{Definition}
\theoremstyle{remark}

\title[Irregular Diophantine quadruples]{There Are Only Finitely Many Irregular Diophantine Quadruples}
\author{Bo He}
\address[Bo He]{Mathematisches Institut, Georg-August-Universit\"at G\"ottingen, 37073 G\"ottingen, Germany}
\email[Bo He]{bo.he@mathematik.uni-goettingen.de}
\address[Bo He]{Institute of Applied Mathematics, Aba Teachers University, Wenchuan, Sichuan 623000, China}
\email[Bo He]{bhe@live.cn}
\author{Alain Togb\'e}
\address[Alain Togb\'e]{Purdue University Northwest, Westville, Indiana 46391, United States of America}
\email[Alain Togb\'e]{atogbe@pnw.edu}
\date{}


\begin{document}

\begin{abstract}
A set of $m$ distinct positive integers is called a Diophantine
$m$-tuple if the product of any two distinct elements, increased by one,
is a perfect square. A Diophantine quadruple $\{a,b,c,d\}$ with
$a<b<c<d$ is called regular if $$d=a+b+c+2abc+2\sqrt{(ab+1)(ac+1)(bc+1)}.$$
We prove that every irregular Diophantine quadruple satisfies
$c<10^{338}$ and $d<10^{10^{24}}$.
\end{abstract}

\maketitle

\section{Introduction}

A set of $m$ distinct positive integers $\{a_1,a_2,\ldots,a_m\}$ is
called a \emph{Diophantine $m$-tuple} if $a_i a_j+1$ is a square for
all $1\leq i<j\leq m$. The problem originates with Diophantus, who
exhibited the rational quadruple $\{1/16,33/16,17/4,105/16\}$.
Fermat later found the first integral example $\{1,3,8,120\}$.
See \cite{Diophantus,DicksonHistory}. Euler showed in the eighteenth
century that infinitely many integral Diophantine quadruples exist.

Every Diophantine triple $\{a,b,c\}$ with $a<b<c$ has the regular extension
\begin{equation}\label{eq:dplus-introduction}
 d_+(a,b,c)=a+b+c+2abc
 +2\sqrt{(ab+1)(ac+1)(bc+1)}.
\end{equation}
This construction is due to Arkin, Hoggatt and Straus
\cite{ArkinHoggattStraus}. A quadruple $\{a,b,c,d\}$ with $a<b<c<d$ is
called \emph{regular} if $d=d_+(a,b,c)$. Equivalently, it satisfies
\begin{equation}\label{eq:regularidentity}
 (a+b-c-d)^2=4(ab+1)(cd+1).
\end{equation}
For the structure of regular quadruples, see
\cite{AssafGueron,GibbsRegular}.
Regarded as a quadratic equation in $d$, \eqref{eq:regularidentity}
has the other root
\begin{equation}\label{eq:dminus-introduction}
 d_-(a,b,c)=a+b+c+2abc
 -2\sqrt{(ab+1)(ac+1)(bc+1)}.
\end{equation}
The triple $\{a,b,c\}$ with $a<b<c$ is called an \emph{Euler triple}
if $c=a+b+2\sqrt{ab+1}$, or equivalently if $d_-(a,b,c)=0$.

The following conjecture was stated explicitly by Dujella and Peth\H{o}
\cite{DujellaPetho}.

\begin{conjecture}[Regularity conjecture]\label{conj:regularity}
Every Diophantine quadruple $\{a,b,c,d\}$ with $a<b<c<d$
satisfies $d=d_+(a,b,c)$.
\end{conjecture}

The regularity conjecture implies that there is no Diophantine
quintuple, since a quintuple would give two distinct larger extensions
of the same triple. Dujella first proved that a Diophantine tuple has at
most eight elements \cite{DujellaAbsolute}. He then proved that there
are no sextuples and only finitely many quintuples
\cite{DujellaFiniteness}. Fujita showed that any quintuple would contain
a regular quadruple \cite{FujitaQuintupleRegular}.
The authors and Ziegler proved that no quintuple exists \cite{HeTogbeZiegler}.

The nonexistence of quintuples does not settle the regularity conjecture.
For a fixed triple $\{a,b,c\}$, it excludes two distinct extensions
$d_1,d_2>c$ for which $d_1d_2+1$ is a square. It leaves open the
possibility of another extension $d\ne d_+(a,b,c)$ when
$d\,d_+(a,b,c)+1$ is not a square.

We prove the following effective finiteness theorem.

\begin{theorem}\label{thm:main}
There are only finitely many irregular Diophantine quadruples.
More precisely, every irregular Diophantine quadruple $\{a,b,c,d\}$
with $a<b<c<d$ satisfies
\[
 c<10^{338},\qquad d<10^{10^{24}}.
\]
\end{theorem}

The descent introduced by the authors and Ziegler \cite{HeTogbeZiegler} replaces
the largest element of a non-Euler triple by $d_-$.
Repeating this step reaches an Euler triple. The number of steps is
called the \emph{degree} of the triple, so Euler triples have degree
zero. The precise definition and its basic properties are recalled in
Section~\ref{sec:descent}.
For triples admitting an irregular extension, the proof gives
$c<10^{98}$ in degree zero, $c<10^{162}$ in degree one, and
$c<10^{338}$ in higher degree. These bounds are not optimized.

The first important uniqueness result predates the conjecture. Baker and
Davenport \cite{BakerDavenport} proved in 1969 that $120$ is the unique
integer exceeding $8$ that extends $\{1,3,8\}$. Dujella
\cite{DujellaParametric} obtained the corresponding result for
$\{k-1,k+1,4k\}$ in 1997, and Dujella and Peth\H{o} \cite{DujellaPetho}
treated $\{1,3,c\}$ in 1998. Table~\ref{tab:parametric-regularity} lists further
cases. A pair in the table specifies the two smallest elements of the
quadruple. Parameters are restricted to values giving distinct positive
integers, and the pair or triple is assumed to be Diophantine.

\begin{table}[htbp]
\caption{Selected cases of the regularity conjecture.}
\label{tab:parametric-regularity}
\centering
\begingroup
\normalsize
\renewcommand{\arraystretch}{1.12}
\setlength{\tabcolsep}{4pt}
\begin{tabular}{@{}>{\raggedright\arraybackslash}p{0.66\linewidth}c@{}}
\hline
Pair or triple and restrictions & References\\
\hline
$\{k-1,k+1\}$
  & \cite{BugeaudDujellaMignotte,FujitaPair}\\
$\{k+1,4k,9k+3\}$
  & \cite{HeTogbeK4K9K}\\
$\{a,b\}$, $b<a+4\sqrt a$
  & \cite{FilipinFujitaTogbeExamples}\\
$\{k,4k\pm4\}$
  & \cite{FilipinFujitaTogbeExamples,HeEtAlPair}\\
$\{k,A^2k\pm2A,(A+1)^2k\pm2(A+1)\}$,\newline
$A,k\geq1$
  & \cite{CipuFujitaMignotte}\\
$\{1,b\}$, $b-1$ a prime power
  & \cite{HePinterTogbeYang}\\
$\{KA^2,4KA^4\pm4A\}$, $A\geq1$, $K\in\{1,2,3,4\}$
  & \cite{CipuFilipinFujitaPairs}\\
\hline
\end{tabular}
\endgroup
\end{table}

Ad\v{z}aga, Filipin and Jurasi\'c \cite{AdzagaFilipinJurasic} proved the
regularity conjecture for quadruples whose two smallest elements are
$2$ and $b$, with $b/2-1$ prime. The authors, jointly with Ad\'edji
and Pint\'er \cite{AdedjiHePinterTogbe}, obtained the corresponding
result for the pairs $\{a,3a\}$ and $\{a,8a\}$.

Earlier results of the authors \cite{HeTogbeFamilyI,HeTogbeFamilyII}
were extended by the two-parameter theorem of Cipu, Fujita and Mignotte
\cite{CipuFujitaMignotte}. Other results concern Fibonacci
and Pell families
\cite{DujellaHoggattBergum,ParkLeeFibonacci,AdedjiFilipinTogbePell},
canonical triples beginning with $KA^2$ and $4KA^4\pm4A$ for arbitrary
positive $A,K$ \cite{CipuFilipinFujitaFamily}, further ranges of pairs
\cite{CipuFilipinFujitaPairs}, and restricted three-parameter families
\cite{CipuFujitaMignotteThree}.

A second line of work concerns the number of larger extensions of a fixed triple.
Fujita and Miyazaki \cite{FujitaMiyazaki} obtained a bound of eleven
larger extensions, subsequently reduced to eight by Cipu, Fujita and
Miyazaki \cite{CipuFujitaMiyazaki}. Irregular extensions of a fixed pair
or triple were studied in \cite{FujitaFixed}, and extensions by smaller
elements in \cite{CipuDujellaFujita}.

Sections~\ref{sec:descent}--\ref{sec:congruences}
recall the descent and the associated Pell equations and
congruences. In Sections~\ref{sec:euler}--\ref{sec:higher-degree}, we
obtain lower bounds for the Pell indices in terms of $c$.
These bounds, combined with the common logarithmic estimate in
Section~\ref{sec:linear-form}, give the upper bounds for $c$ and $d$.

\section{Descent and degree}\label{sec:descent}

We use the descent introduced in \cite{HeTogbeZiegler}.
Let $T=\{a,b,c\}$ be a Diophantine triple with $a<b<c$, and write
\[
 ab+1=r^2,\qquad ac+1=s^2,\qquad bc+1=t^2,
\]
where $r,s,t>0$. Throughout the paper we put
\[
 \alpha=r+\sqrt{ab},\qquad
 \beta=s+\sqrt{ac},\qquad
 \gamma=t+\sqrt{bc}.
\]
Write $d_{\pm}(T)=d_{\pm}(a,b,c)$ for the regular roots introduced
in Section~1. A direct calculation gives
\begin{equation}\label{eq:dpm-squares}
\begin{aligned}
 ad_{\pm}+1&=(rs\pm at)^2,\\
 bd_{\pm}+1&=(rt\pm bs)^2,\\
 cd_{\pm}+1&=(cr\pm st)^2.
\end{aligned}
\end{equation}
Moreover, $0\leq d_-<c<d_+$. If $d_->0$, the square identities
exclude $d_-=a$ and $d_-=b$, since neither $a^2+1$ nor $b^2+1$
is a square. Thus $\{a,b,d_-\}$ is again a Diophantine triple.

\begin{lemma}[{\cite[Lemmas~1 and~2]{HeTogbeZiegler}}]\label{lem:basic-descent}
Let $\{a,b,c\}$ be a Diophantine triple with $a<b<c$.
Then either $c=a+b+2r$ or $c>4ab$. Moreover,
\[
 4abc+c<d_+(a,b,c)<4abc+4c.
\]
\end{lemma}

The regular extension $d_+=d_+(T)$ gives three new Diophantine triples
$\{a,b,d_+\}$, $\{a,c,d_+\}$ and $\{b,c,d_+\}$.
The inverse-root identities in
\cite[Propositions~1 and~2]{HeTogbeZiegler} give
\[
 d_-(a,b,d_+)=c,\qquad d_-(a,c,d_+)=b,\qquad d_-(b,c,d_+)=a.
\]
The descent reverses this construction. For a non-Euler triple, put
$e=d_-(T)>0$. The same identities give
\[
 c=d_+(a,b,e).
\]
Thus replacing $c$ by $e$ gives a Diophantine triple with smaller
product of its elements. Define
\begin{equation}\label{eq:partial-definition}
 \partial T=(T\cup\{e\})\setminus\{\max T\}.
\end{equation}
Starting with $\partial^0T=T$, define
\[
 \partial^jT=\partial(\partial^{j-1}T),\qquad j\geq1,
\]
as long as $\partial^{j-1}T$ is not Euler. We also put
\begin{equation}\label{eq:predecessor-elements}
 d_{-j}(T)=d_-(\partial^{j-1}T),\qquad j\geq1,
\end{equation}
whenever $\partial^{j-1}T$ is defined. In particular, $d_{-1}(T)=d_-(T)$.

\begin{proposition}[{\cite[Proposition~3]{HeTogbeZiegler}}]
\label{prop:finite-descent}
For every Diophantine triple $T=\{a,b,c\}$, the descent reaches an
Euler triple after a unique number $j$ of steps, where
\[
 0\leq j<\frac{\log(abc)}{\log12}.
\]
At each step, the product of the three elements decreases by a factor
greater than $12$.
\end{proposition}

\begin{proof}
For a non-Euler triple $T=\{a,b,c\}$ with $a<b<c$, put $e=d_-(T)$.
The inverse-root identity and Lemma~\ref{lem:basic-descent} give
$c=d_+(a,b,e)>4abe$. Since $ab\geq3$, we have
\[
 abe<\frac c4\leq\frac{abc}{12}.
\]
This estimate applies at each step after ordering the current triple.
The product is a positive integer, so the descent must terminate.
If it takes $j>0$ steps, the final product is less than $abc/12^j$,
and hence $12^j<abc$. The stated bound also holds for $j=0$.
The terminal step is unique because the descent is defined only for
non-Euler triples.
\end{proof}

\begin{definition}\label{def:degree}
The integer $j$ in Proposition~\ref{prop:finite-descent} is called the
\emph{degree} of $T$, denoted by $\deg T$. The Euler triple
$\partial^jT$ is called the \emph{generator} of $T$, and we say that
$T$ is generated by $\partial^jT$.
\end{definition}

Thus Euler triples have degree zero. If $\deg T=j$, then
\[
 d_{-k}(T)>0\quad(1\leq k\leq j),\qquad d_{-(j+1)}(T)=0.
\]
Each of the three successors $\{a,b,d_+\}$, $\{a,c,d_+\}$ and
$\{b,c,d_+\}$ has predecessor $T$ and degree $j+1$.
Conversely, the inverse-root identities show that every triple with
predecessor $T$ is one of these three successors. Hence each Euler
triple is the root of a ternary tree, and it generates at most $3^j$
triples of degree $j$ \cite[Remark~4.2]{HeTogbeZiegler}.
For example, the Fermat triple $\{1,3,8\}$ generates the degree-one triples
\[
 \{1,3,120\},\qquad \{1,8,120\},\qquad \{3,8,120\}.
\]

For later use, we record the explicit form of a degree-one triple
\cite[proof of Theorem~8]{HeTogbeZiegler}.

\begin{proposition}\label{prop:degree-one-form}
If $\deg T=1$, then $\{a,b,e\}$ is Euler, $e=a+b\pm2r$, and
\begin{equation}\label{eq:degree-one-c}
 c=d_+(a,b,e)=4r(r\pm a)(b\pm r),
\end{equation}
where the signs are correlated.
\end{proposition}

For higher degree, the first two elements in
\eqref{eq:predecessor-elements} are
\[
 e=d_{-1}(T)=d_-(T)>0,\qquad
 f=d_{-2}(T)=d_-(\partial T)>0.
\]
The inverse-root identity applies again at this second descent step.
According as $e<b$ or $e>b$, it gives $b=d_+(a,e,f)$ or $e=d_+(a,b,f)$.

We shall use the following two restrictions on an irregular quadruple.

\begin{lemma}\label{lem:irregular-above-regular}\label{lem:higher-b-absolute}
Every irregular Diophantine quadruple $\{a,b,c,d\}$ with $a<b<c<d$
satisfies
\[
 d>d_+(a,b,c),\qquad b>4000.
\]
\end{lemma}

\begin{proof}
Since $c\geq3$, \cite[Proposition~1]{DujellaFiniteness} and
Lemma~\ref{lem:basic-descent} give
\[
 d>2.695c^{7/2}a^{5/2}>4c^3
 >4c(ab+1)>d_+(a,b,c).
\]
Here $ab+1<c^2$, since $a<b<c$ are positive integers.
By \cite[Lemma~3.4]{CipuFilipinFujitaBoundsII}, one has $b>21000$
if $b<2a$, $b>130000$ if $2a\leq b\leq12a$, and $b>4000$ if $b>12a$.
\end{proof}

\section{Pell equations}\label{sec:pell}

Suppose that $\{a,b,c,d\}$ is a Diophantine quadruple with $a<b<c<d$.
Retain $r,s,t$ from Section~\ref{sec:descent}, and write
\[
 ad+1=x^2,\qquad bd+1=y^2,\qquad cd+1=z^2
\]
with positive integers $x,y,z$. Eliminating $d$, we obtain the system
of Pell-type equations
\begin{align}
  ay^2-bx^2&=a-b,\label{eq:pell-ab}\\
  az^2-cx^2&=a-c,\label{eq:pell-ac}\\
  bz^2-cy^2&=b-c.\label{eq:pell-bc}
\end{align}

We use the following description of the solutions.

\begin{lemma}[{\cite[Lemma~1]{DujellaFiniteness}}]\label{lem:pell-general}
Let \(A<B\) be positive integers such that \(AB+1=R^2\). Every positive
integer solution $(X,Y)$ of $AY^2-BX^2=A-B$ is obtained from
\begin{equation}\label{eq:general-pell-solution}
  Y\sqrt A+X\sqrt B
  =\bigl(Y_0\sqrt A+X_0\sqrt B\bigr)
   \bigl(R+\sqrt{AB}\bigr)^k,
\end{equation}
where \(k\geq0\) and the initial solution \((X_0,Y_0)\) may be chosen so
that
\[
  1\leq X_0\leq
  \sqrt{\frac{A(B-A)}{2(R-1)}},
  \qquad
  1\leq |Y_0|\leq
  \sqrt{\frac{(R-1)(B-A)}{2A}}.
\]
\end{lemma}

Applying Lemma~\ref{lem:pell-general} to
\eqref{eq:pell-ab}--\eqref{eq:pell-bc}, we obtain initial solutions
$(x_0,y_0)$, $(x_1,z_1)$, and $(y_2,z_2)$. Define their coordinate
sequences, for $j\geq0$, by
\begin{align}
 y^{(a,b)}_j\sqrt a+x^{(a,b)}_j\sqrt b
 &=\bigl(y_0\sqrt a+x_0\sqrt b\bigr)\alpha^j,
 \label{eq:pell-sequence-ab}\\
 z^{(a,c)}_j\sqrt a+x^{(a,c)}_j\sqrt c
 &=\bigl(z_1\sqrt a+x_1\sqrt c\bigr)\beta^j,
 \label{eq:pell-sequence-ac}\\
 z^{(b,c)}_j\sqrt b+y^{(b,c)}_j\sqrt c
 &=\bigl(z_2\sqrt b+y_2\sqrt c\bigr)\gamma^j.
 \label{eq:pell-sequence-bc}
\end{align}
The quadruple gives nonnegative indices $l',m',n'$ such that
\begin{equation}\label{eq:common-coordinates}
  x=x^{(a,b)}_{l'}=x^{(a,c)}_{m'},\qquad
  y=y^{(a,b)}_{l'}=y^{(b,c)}_{n'},\qquad
  z=z^{(a,c)}_{m'}=z^{(b,c)}_{n'}.
\end{equation}

For each pair $(A,B)\in\{(a,b),(a,c),(b,c)\}$, put $R=\sqrt{AB+1}$.
The two coordinate sequences satisfy
\begin{align*}
 X_1&=RX_0+AY_0,&Y_1&=RY_0+BX_0,\\
 X_{j+2}&=2RX_{j+1}-X_j,&Y_{j+2}&=2RY_{j+1}-Y_j.
\end{align*}
The initial coordinates are $(x_0,y_0)$, $(x_1,z_1)$, and $(y_2,z_2)$,
respectively. We reserve $l',m',n'$ for indices of a common solution.

The two representations of $z$ initially give four parity cases
\cite{DujellaFiniteness,FujitaMiyazaki}. The shifts in
\cite{CipuFujitaMiyazaki} reduce them to two.

\begin{lemma}[{\cite[Lemma~8]{DujellaFiniteness},
\cite[Theorem~1.3]{FujitaMiyazaki},
\cite[Lemma~2.3]{CipuFujitaMiyazaki}}]
\label{lem:solution-types}
Suppose that \(\{a,b,c,d\}\) is a Diophantine quadruple with
$a<b<c<d$, and that $z=z^{(a,c)}_{m'}=z^{(b,c)}_{n'}$
for nonnegative integers \(m'\) and \(n'\). The initial solutions may be
chosen so that, for some $\lambda\in\{\pm1\}$, one of the following holds.
\begin{enumerate}[label=\textup{(\roman*)},leftmargin=2.5em]
\item $m'=2m$, $n'=2n$, and $z_1=z_2=\lambda$.
\item $m'=2m+1$, $n'=2n+1$, and $(z_1,z_2)=(\lambda t,\lambda s)$.

\end{enumerate}
\end{lemma}

\begin{proof}
The cited results show that, before normalization, the
common term belongs to the even--even, odd--even,
even--odd, or odd--odd case.  In the even--even case the initial terms are either
\((z_1,z_2)=(\lambda,\lambda)\), or
\((|z_1|,|z_2|)=(|cr-st|,|cr-st|)\). In the other three cases the absolute values of
the initial terms are $(t,|cr-st|)$, $(|cr-st|,s)$, and $(t,s)$,
respectively, with the sign relations specified in the cited theorem.

To see how these four cases reduce to two, write
\(z^{(a,c)}_k(\zeta)\) for the \((a,c)\)-sequence whose initial term is
\(\zeta\), and use the analogous notation for the \((b,c)\)-sequence.
By \cite[Lemma~2.3]{CipuFujitaMiyazaki},
\[
\begin{aligned}
 z^{(a,c)}_k(cr-st)&=z^{(a,c)}_{k+1}(-t),&
 z^{(a,c)}_{k+1}(st-cr)&=z^{(a,c)}_k(t),\\
 z^{(b,c)}_k(cr-st)&=z^{(b,c)}_{k+1}(-s),&
 z^{(b,c)}_{k+1}(st-cr)&=z^{(b,c)}_k(s).
\end{aligned}
\]
Thus the nontrivial even--even case and both mixed-parity cases can be
reindexed as an odd--odd coincidence with initial terms
\((\lambda t,\lambda s)\).  The remaining even--even case has
initial terms \((\lambda,\lambda)\).  This gives precisely
\textup{(i)} and \textup{(ii)}.
The reverse shifts preserve nonnegative indices, since an original index zero
with initial term $\pm(cr-st)$ would give $d=d_-<c$, contrary to $d>c$.
\end{proof}

We call the cases in Lemma~\ref{lem:solution-types}\textup{(i)} and
\textup{(ii)} \emph{type I} and \emph{type II}, respectively.
Their common-term equations are
\begin{align}
 z^{(a,c)}_{2m}&=z^{(b,c)}_{2n} &&\text{(type I)},\label{eq:type-I-common-term}\\
 z^{(a,c)}_{2m+1}&=z^{(b,c)}_{2n+1} &&\text{(type II)}.\label{eq:type-II-common-term}
\end{align}
In both types we write $\Delta=m-n$.

\section{Congruences and index bounds}\label{sec:congruences}

The coordinate recurrences give the following congruences.

\begin{lemma}[{\cite[Lemma~9]{DujellaFiniteness}}]\label{lem:quadratic-expansions}
For every nonnegative integer $m$,
\begin{align}
 z^{(a,c)}_{2m}
 &\equiv z_1+2cm(sx_1+amz_1)\pmod{8c^2},
 \label{eq:z-even-general-c2}\\
 z^{(a,c)}_{2m+1}
 &\equiv sz_1+c\bigl((2m+1)x_1+2am(m+1)sz_1\bigr)
 \pmod{4c^2}.
 \label{eq:z-odd-general-c2}
\end{align}
The corresponding formulas for the $(b,c)$-sequence follow by replacing
$(a,s,m,x_1,z_1)$ with $(b,t,n,y_2,z_2)$.
\end{lemma}

\begin{proof}
Each parity subsequence satisfies
\[
 Z_{m+2}=(4ac+2)Z_{m+1}-Z_m.
\]
Direct substitution shows that the two quadratic expressions in the lemma
satisfy this recurrence modulo $8c^2$ and $4c^2$, respectively.
They agree with the corresponding coordinate sequences at $m=0,1$
modulo these same moduli. Induction proves both congruences.
\end{proof}

\begin{lemma}\label{lem:mn-comparison}
For the normalized choices of initial solutions in
Lemma~\ref{lem:solution-types}, a common term
$z=z^{(a,c)}_{m'}=z^{(b,c)}_{n'}$ corresponding to $d>c$ satisfies
$n'-1\leq m'\leq2n'+1$. Consequently,
\[
 \begin{cases}
   n\leq m\leq2n, & \text{for type I},\\
   n\leq m\leq2n+1, & \text{for type II}.
 \end{cases}
\]
\end{lemma}

\begin{proof}
For type I, the initial solutions are $(x_1,z_1)=(1,\lambda)$ and
$(y_2,z_2)=(1,\lambda)$. The first assertion is therefore
\cite[Lemma~3]{DujellaFiniteness}, whose even-index specialization gives
$n\leq m\leq2n$.

For type II, put $V_j=z^{(a,c)}_{2j+1}$, $W_j=z^{(b,c)}_{2j+1}$,
and $h=cr+\lambda st>0$. The last inequality follows from
$(cr)^2-(st)^2=c(c-a-b)-1>0$, since $c>a+b$.
These increasing sequences have recurrence coefficients
$p=4ac+2<q=4bc+2$, and
\[
 V_0=W_0=h,\qquad V_1=(p-1)h+2cr<W_1=(q-1)h+2cr.
\]
For positive sequences $U_j,Z_j$ with coefficients $P>Q>2$, the difference
$D_j=U_j-Z_j$ satisfies
$D_{j+1}-D_j=(P-2)D_j+(D_j-D_{j-1})+(P-Q)Z_j$.
Thus positive, increasing differences propagate by induction.
Applied with $(U,Z)=(W,V)$, this gives $V_j<W_j$ for $j\geq1$, so $m>n$
if $n\geq1$, while $n=0$ forces $m=0$.

For the other direction, $C_j=V_{2j+1}$ has coefficient
$P=p^2-2>q$, initial term $C_0=V_1>W_0$, and
\[
 (C_1-W_1)-(C_0-W_0)=(P-1)V_1-(q-1)h>0.
\]
Applied with $(U,Z)=(C,W)$, the same comparison gives
$W_n<C_n=V_{2n+1}$, hence $m\leq2n$.
This proves both asserted bounds.
\end{proof}

\begin{lemma}\label{lem:constant-index-bounds}
Let $d>d_+(a,b,c)$. Then $m'>2$, and a type-I solution has
$m\geq2$, $n\geq1$. If a type-II solution also satisfies
$b>4000$ and $c>4ab+a+b$, then $m,n\geq4$.
Moreover, the index $l'$ in \eqref{eq:common-coordinates} is positive.
\end{lemma}

\begin{proof}
By \cite[Lemma~2.5 and the proof of Theorem~1.4]{CipuFujitaMiyazaki},
a common solution with $m'\leq2$ gives
$d\in\{0,d_{-},d_{+}\}$. Hence $m'>2$. For type I this gives
$m\geq2$, while Lemma~\ref{lem:mn-comparison} gives $n\geq1$.
For type II, \cite[Lemma~2.6]{CipuFujitaMiyazaki} yields
$\min\{m',n'\}\geq9$, and therefore $m,n\geq4$.
Finally, $l'=0$ would imply $x=x_0$. By
Lemma~\ref{lem:pell-general},
\[
 d=\frac{x_0^2-1}{a}
 \leq\frac{b-a}{2(r-1)}-\frac1a<b<c,
\]
contrary to the hypothesis.
\end{proof}

\section{Euler triples}\label{sec:euler}

Let $\{a,b,c\}$ be an Euler triple. Then
\[
 c=a+b+2r,\qquad s=a+r,\qquad t=b+r,\qquad cr-st=1.
\]

\begin{lemma}\label{lem:euler-type-I}\label{lem:euler-ab-sequence}
After reindexing, every common solution has the representations
\begin{align}
 y\sqrt a+x\sqrt b
   &=(\sqrt a+\sqrt b)\alpha^{2l},
   \label{eq:euler-ab-normalized}\\
 z\sqrt a+x\sqrt c
   &=(\lambda\sqrt a+\sqrt c)\beta^{2m},\label{eq:euler-ac-seq}\\
 z\sqrt b+y\sqrt c
   &=(\lambda\sqrt b+\sqrt c)\gamma^{2n},\label{eq:euler-bc-seq}
\end{align}
where $l,m,n\geq0$ and $\lambda\in\{\pm1\}$.
\end{lemma}

\begin{proof}
Since $cr-st=1$, the shifts in the proof of
Lemma~\ref{lem:solution-types} convert type II to type I. Initial terms $(-t,-s)$ at indices $(2m+1,2n+1)$ become $(1,1)$
at $(2m,2n)$, while $(t,s)$ become $(-1,-1)$ at $(2m+2,2n+2)$.
Thus the last two representations hold after reindexing.

For the $(a,b)$-equation, choose a reduced solution $(x_0,y_0)$
and exponent $l'$ as in Lemma~\ref{lem:pell-general}.
The even $(b,c)$-representation gives $y\equiv1\pmod b$, hence
$y_0\equiv r^{l'}\pmod b$. Moreover
$|y_0|^2\leq b(b-a)/(2(r+1))<b^2/4$.
If $l'$ is even, then $y_0=x_0=1$.

Suppose now that $l'$ is odd. Then $y_0=r$ or $r-b$. The former would give
$x_0^2=a^2+1$, while the latter gives $x_0=r-a$ and
$y_0\sqrt a+x_0\sqrt b=(\sqrt a+\sqrt b)\alpha^{-1}$.
In either case the final exponent is even, proving the first representation.
\end{proof}

Write $X_j,Y_j$ for the coordinates of
$(\sqrt a+\sqrt b)\alpha^j$.

\begin{lemma}\label{lem:euler-index-order}
For an irregular extension of an Euler triple,
\[
 n\leq m<l<2n,\qquad
 l\log\alpha<m\log\beta<n\log\gamma.
\]
In particular,
\begin{equation}\label{eq:euler-index-difference}
 m>(l-m)\sqrt{b/a}\log\alpha.
\end{equation}
\end{lemma}

\begin{proof}
Lemmas~\ref{lem:irregular-above-regular},
\ref{lem:constant-index-bounds}, and \ref{lem:mn-comparison}
give $m\geq2$ and $n\leq m$.
We use the recurrence comparison from the proof of
Lemma~\ref{lem:mn-comparison}.
For positive sequences $A_j,B_j$ with recurrence coefficients
$2<P<Q$, a common zeroth term and $B_1\geq A_1>A_0$,
it gives $B_j>A_j$ for $j\geq2$.

Apply this to $A_j=X_{2j}$ and $B_j=x^{(a,c)}_{2j}$.
Their coefficients are $4ab+2<4ac+2$, and
$A_0=B_0=1$, $A_1=1+2a(b+r)$,
$B_1=1+2a(c+\lambda s)\geq A_1$, since $c-s=b+r$.
Thus $x=X_{2l}>X_{2m}$ and $l>m$.
For $C_j=Y_{4j}$ and $E_j=y^{(b,c)}_{2j}$, we have $C_0=E_0=1$
and $E_1=1+2b(c+\lambda t)>1$, since $c-t=a+r$.
Moreover,
\[
 \frac{C_1-E_1}{2b}
 \geq(4r^2-1)a-2b+4r^3-5r
 \geq2r^2+1+4r^3-5r>0,
\]
using $b\leq r^2-1$ and $a\geq1$.
Also $\alpha^2>2t>\gamma$, since
$\alpha^2-2t=2b(a-1)+1+2r(\sqrt{ab}-1)>0$.
Their recurrence coefficients therefore satisfy
$\alpha^4+\alpha^{-4}>\gamma^2+\gamma^{-2}$.
The same comparison gives $Y_{4n}>y^{(b,c)}_{2n}=Y_{2l}$, hence $l<2n$.

For the logarithmic inequalities, put $p=(1+\sqrt{a/b})/2$ and $q=(1+\lambda\sqrt{a/c})/2$.
The common $x$ equals $f_p(\alpha^{2l})=f_q(\beta^{2m})$,
where $f_p(u)=pu+(1-p)u^{-1}$.
Since $0<q<p<1$, $p>1/2$, and $f_p$ increases on $[1,\infty)$,
we obtain $\alpha^{2l}<\beta^{2m}$.
For $z$, the corresponding functions are
$Au+(\lambda-A)u^{-1}$, with
$A=(\lambda+\sqrt{c/a})/2>B=(\lambda+\sqrt{c/b})/2>0$.
Both increase, and the first exceeds the second for $u>1$, so $\beta^{2m}<\gamma^{2n}$.
Finally,
\[
 \log\beta-\log\alpha
 =\int_r^{r+a}\frac{du}{\sqrt{u^2-1}}
 <\sqrt{a/b},
\]
which gives \eqref{eq:euler-index-difference}.
\end{proof}

Here $c/b$ may be close to $1$, so a bound of order $\sqrt{c/b}$
does not give the required growth of $m$. We compare the three
representations in
\eqref{eq:euler-ab-normalized}--\eqref{eq:euler-bc-seq}
modulo $8a^2$ and $8b^2$.
The polynomial $1+2a(cj^2+\lambda sj)$ has the same first two terms
as $x^{(a,c)}_{2j}$, and its residual in that recurrence is
$-8a^2c[c(j+1)^2+\lambda s(j+1)]$.
It therefore agrees with the coordinate modulo $8a^2$.
Applying the same calculation to the other coordinates in
\eqref{eq:euler-ab-normalized}--\eqref{eq:euler-bc-seq}, and dividing by
$2a$ or $2b$, gives
\begin{align}
 bl^2+rl&\equiv cm^2+\lambda sm\pmod{4a},\label{eq:euler-cong-a}\\
 al^2+rl&\equiv cn^2+\lambda tn\pmod{4b}.\label{eq:euler-cong-b}
\end{align}

\begin{proposition}\label{prop:euler-m5}
For an irregular extension of an Euler triple with $c\geq10^7$, one has
\[
 m^5>\frac{c\log c}{32}.
\]
\end{proposition}

\begin{proof}
Since $a+b\leq r^2$ and $r\geq2$, we have
$c\leq r^2+2r<\alpha^2$, so $\log\alpha>\tfrac12\log c>8$.
Lemma~\ref{lem:euler-index-order} gives $m>\log\alpha>8$ and $n\geq5$.
We first show that $r<4m^4$.

Put
\begin{align*}
 M&=2m^2+\lambda m-l,& N&=2n^2+\lambda n-l,\\
 L_1&=l^2-m^2,& L_2&=l^2-n^2.
\end{align*}
By \eqref{eq:euler-cong-a} and \eqref{eq:euler-cong-b}, we have
$bL_1\equiv rM\pmod a$ and $aL_2\equiv rN\pmod b$.
Thus $K=(rN-aL_2)/b$ is an integer, and
\[
 rNL_1\equiv bKL_1\equiv rKM\pmod a.
\]
Since $(a,r)=1$, this gives $NL_1\equiv KM\pmod a$.
Multiplying by $b$, we obtain
\begin{equation}\label{eq:euler-master-congruence}
 S:=aML_2+bNL_1\equiv rMN\pmod{ab}.
\end{equation}

Suppose that $4m^4\leq r$. Since $m<l<2n\leq2m$, we have
$0<M<2m^2$, $0<N<2n^2$, and $L_1,L_2>0$.
Hence $MN<4m^2n^2\leq r$ and $0<rMN\leq r^2-r<ab$.
As $S>0$, \eqref{eq:euler-master-congruence} implies $S\geq rMN$.

If $aML_2<3bNL_1$, then $rM<4bL_1$.
Using \eqref{eq:euler-index-difference} and $M\geq2m^2-3m$, we obtain
\[
 \log\alpha<\frac{4m(l+m)}{M}
 \leq\frac{12m}{2m-3}\leq\frac{36}{5}<8,
\]
a contradiction.
Otherwise $S\leq(4/3)aML_2$, so
\[
 \frac ra\leq\frac{4L_2}{3N}\leq\frac{4n}{2n-3}\leq\frac{20}{7}<3.
\]
This gives $b/a<9$.
Since $\gamma/\beta<\sqrt{b/a+1/(ac)}<\sqrt{10}$, the logarithmic
comparison gives $m/n<1+\log10/16<7/6$.
Also \eqref{eq:euler-index-difference} gives $l/m<9/8$, hence $l/n<4/3$.
As $n\geq5$, it follows that
\[
 L_2<\frac79n^2,\qquad N>\frac{23}{15}n^2,
 \qquad \frac ra<\frac{140}{207}<1.
\]
This contradicts $r^2=ab+1>a^2$, proving $r<4m^4$.

Finally, $c/r=a/r+b/r+2<3+\sqrt{b/a}$, while
Lemma~\ref{lem:euler-index-order} gives both
$\log\alpha<m$ and $\sqrt{b/a}\log\alpha<m$.
Therefore
\[
 c\log\alpha<4rm<16m^5.
\]
Together with $\log\alpha>\tfrac12\log c$, this proves the assertion.
\end{proof}

\section{Degree-one triples}\label{sec:degree-one}

Let $T=\{a,b,c\}$ have degree one and put $e=d_-(T)$. By
Proposition~\ref{prop:degree-one-form}, there is a sign
$\varepsilon\in\{\pm1\}$ such that
\begin{equation}\label{eq:degree-one-structure}
 e=a+b+2\varepsilon r,\qquad
 c=4r(r+\varepsilon a)(b+\varepsilon r).
\end{equation}

\begin{lemma}\label{lem:degree-one-size}
For a degree-one triple one has
\[
 b^2<c<\min\{16b^3,32ab^2\},\qquad c>4ab+a+b.
\]
\end{lemma}

\begin{proof}
Since $a<r<b$ and $r^2<2ab$, the formula for $c$ gives
\[
 c<16r^2b<\min\{16b^3,32ab^2\}.
\]
Also $c=d_+(a,b,e)>4abe+a+b+e>4ab+a+b$.

For the lower bound $c>b^2$, first suppose that $\varepsilon=1$.
Then $e>b$, and $c>4abe>b^2$.
Suppose next that $\varepsilon=-1$, and put $u=r-a\geq2$.
Then $u^2=ae+1$, $b=a+e+2u$, and $a+e\leq ae+1=u^2$, so
\[
 b\leq u^2+2u<4(u^2-1)=4ae.
\]
Again $c=d_+(a,b,e)>4abe$, and therefore $c>b^2$.
\end{proof}

\begin{proposition}\label{prop:degree-one-m-lower}
Every degree-one triple admitting an irregular extension satisfies
\[
 m>c^{1/8}.
\]
\end{proposition}

\begin{proof}
Assume to the contrary that $m\leq c^{1/8}$. Then $n\leq m$ by
Lemma~\ref{lem:mn-comparison}.
Lemma~\ref{lem:irregular-above-regular} gives $b>4000$.
By Lemma~\ref{lem:degree-one-size}, $c>b^2>4000^2$ and $c>4ab+a+b$.
In particular $c^{1/8}>7$, which suffices for the size comparisons below.

\smallskip
\noindent\emph{Type I.}
Here $z_1=z_2=\lambda$ and $x_1=y_2=1$.
The even congruence in Lemma~\ref{lem:quadratic-expansions} gives
\[
 am^2+\lambda sm\equiv bn^2+\lambda tn\pmod{4c}.
\]
Since $a<b<c^{1/2}$ and $s<t<\sqrt2\,c^{3/4}$, we have
\[
 \bigl|am^2+\lambda sm-bn^2-\lambda tn\bigr|
 <2c^{3/4}+2\sqrt2\,c^{7/8}<c.
\]
Thus
\[
 am^2+\lambda sm=bn^2+\lambda tn.
\]
Put $D_0=bn^2-am^2$. Then $D_0=\lambda(sm-tn)$.
Squaring the last equality and using $s^2=ac+1$, $t^2=bc+1$, gives
\[
 m^2-n^2=D_0\bigl(c+\lambda(tn+sm)\bigr).
\]
If $D_0=0$, then $m=n$ and $(b-a)m^2=0$, contrary to $m\geq2$.
Otherwise, as $tn+sm<2\sqrt2\,c^{7/8}<c/2$, the right-hand side
has absolute value greater than $c/2$, contradicting
$|m^2-n^2|\leq m^2\leq c^{1/4}$.

\smallskip
\noindent\emph{Type II.}
Here $z_1=\lambda t$, $z_2=\lambda s$, and $x_1=y_2=r$.
The hypotheses of Lemma~\ref{lem:constant-index-bounds} now hold,
so $m,n\geq4$.
Put $\sigma=\lambda\varepsilon$.
The odd congruence in Lemma~\ref{lem:quadratic-expansions},
multiplied by $\lambda st$, gives
\[
 bn(n+1)-am(m+1)\equiv\lambda rst\Delta\pmod c.
\]
Since $2rst=c(2r^2-1)-2\varepsilon r$ and $4\mid c$, this becomes
\[
 C:=bn(n+1)-am(m+1)+\sigma r\Delta
 \equiv\lambda(c/2)\Delta\pmod c.
\]
The bounds $a<b<c^{1/2}$, $r<b$, and $n\leq m$ give
\[
 |C|<\frac52c^{3/4}+c^{5/8}<\frac c2.
\]
If $\Delta$ were odd, the congruence would give $|C|\geq c/2$.
Hence $\Delta$ is even, and $c\mid C$ forces $C=0$.

Set
\begin{align*}
 M&=2m+1,& N&=2n+1,\\
 P&=aM-\sigma r,& Q&=bN-\sigma r.
\end{align*}
Expanding $C=0$ gives $aQ^2-bP^2=a-b$.
Hence
\[
 D=\frac{P^2-1}{a}=\frac{Q^2-1}{b}
  =aM^2-2\sigma rM+b
\]
is an integer. As $N\geq9$ and $r<b$, we have $Q>8b$ and $D>b$,
so $\{a,b,D\}$ is a Diophantine triple.
On the other hand, $M\leq9m/4$ and $c<32ab^2$ give
\[
 D\leq\frac{81}{16}am^2+\frac92rm+b
  <(23b^{-1/4}+1)ab<4ab.
\]
Indeed, the middle bound uses $r^2<2ab$, $a<b$, and
$(81/16)2^{5/4}+(9/2)2^{9/8}<23$. The last follows from
$b>4000>(23/3)^4$.
Lemma~\ref{lem:basic-descent} therefore forces $D=a+b+2r$.
But then $Q=\sqrt{bD+1}=b+r<2b$, contradicting $Q>8b$.
\end{proof}

\section{Higher-degree triples}\label{sec:higher-degree}

Let $T=\{a,b,c\}$ have degree at least two, and put
\[
 e=d_-(T),\qquad f=d_-(\partial T)>0.
\]
Then $\{a,b,e,c\}$ is regular and
\begin{equation}\label{eq:higher-regular-predecessor}
 c=d_+(a,b,e),\qquad c>4abe.
\end{equation}

\begin{lemma}\label{lem:higher-b8a}
If $\deg T\geq2$ and $T$ admits an irregular extension
$d>d_+(a,b,c)$, then $b>8a$.
\end{lemma}

\begin{proof}
Suppose $b\leq8a$.
By \cite[Lemma~1.3]{FilipinFujitaTogbeExamples}, $b>21000$, so $a>500$.
If $e<b$, the regular-root identities give
$b=d_+(a,e,f)>4aef\geq8a$, since $e,f$ are distinct positive integers.
Thus $e>b$, whence
\[
 e=d_+(a,b,f)>4abf,\qquad c>16a^2b^2f.
\]
If $f\geq40$, this yields $c>10b^4$.

Suppose $f<40$, so $f<a$.
If $\{f,a,b\}$ were not Euler, its positive predecessor $g$ would give
$b=d_+(a,f,g)>4afg\geq8a$, again a contradiction.
Hence $b=a+f+2\sqrt{af+1}$, and
\[
 \frac{b}{a\sqrt f}
 =\frac1{\sqrt f}+\frac{\sqrt f}{a}
   +2\sqrt{\frac1a+\frac1{a^2f}}
 <1+\frac4{\sqrt a}<\frac54.
\]
It follows that
\[
 c>16a^2b^2f>\frac{256}{25}b^4>10b^4.
\]
For $b\leq8a$, both thresholds in \cite[Theorem~1.5]{FujitaMiyazaki}
are below $10b^4$, so $d=d_+$,
contrary to irregularity.
\end{proof}

We use $b>8a$ throughout the rest of this section.
The type-II constant index bounds also apply. Indeed, $b>4000$, and
the regular-root formula with $e\geq1$ gives
\[
 c=a+b+e+2abe+2\sqrt{(ab+1)(ae+1)(be+1)}
 >4abe+a+b+e>4ab+a+b.
\]
Hence Lemma~\ref{lem:constant-index-bounds} gives $m,n\geq4$ for type II.

\subsection{Type I}

Consider an irregular common solution
$z^{(a,c)}_{2m}=z^{(b,c)}_{2n}$ with $z_1=z_2=\lambda\in\{\pm1\}$.
Lemmas~\ref{lem:constant-index-bounds} and \ref{lem:mn-comparison} give
$m\geq2$ and $1\leq n\leq m\leq2n$.
Put $\kappa=(\sqrt{17}-1)/2$, so $\kappa^2+\kappa=4$.

\begin{lemma}\label{lem:higher-type-I-first-lower}
The common indices satisfy $m>n$ and
\begin{equation}\label{eq:higher-type-I-gap}
 n>\kappa\sqrt{c/b}.
\end{equation}
\end{lemma}

\begin{proof}
This is the congruence argument of \cite[Lemma~16]{HeTogbeZiegler}
in the present normalization.
Write $V_j=z^{(a,c)}_{2j}$ and $W_j=z^{(b,c)}_{2j}$.
Their recurrence coefficients are $4ac+2$ and $4bc+2$, with
\[
 V_0=W_0=\lambda,\qquad
 V_1=\lambda+2c(s+\lambda a),\qquad
 W_1=\lambda+2c(t+\lambda b).
\]
Applying the even congruence in Lemma~\ref{lem:quadratic-expansions}
to $V_m=W_n$ gives
\begin{equation}\label{eq:higher-type-I-cong-4c}
 am^2+\lambda sm\equiv bn^2+\lambda tn\pmod{4c}.
\end{equation}

Since $c>4ab\geq4b$, we have $s+t<c$.
Thus $W_1>V_1>0$. For $\lambda=-1$, use
$W_1-V_1=2c(b-a)(c/(s+t)-1)>0$.
The $V_j$ are increasing for $j\geq1$.
For $D_j=W_j-V_j$, we have $D_0=0$, $D_1>0$, and
\[
 D_{j+2}=(4bc+2)D_{j+1}-D_j+4c(b-a)V_{j+1}.
\]
Induction gives $D_{j+1}>D_j$ for $j\geq1$, hence $W_j>V_j$
for every $j\geq1$. This excludes $m=n$.

Equality in \eqref{eq:higher-type-I-cong-4c} would give, by
$s^2=ac+1$ and $t^2=bc+1$,
\[
 m^2-n^2=(bn^2-am^2)\bigl(c+\lambda(tn+sm)\bigr).
\]
The second factor is a nonzero integer, so
$bn^2-am^2\leq m^2-n^2$. But $m\leq2n$ and $b>8a$ give
$bn^2-am^2>4an^2\geq4n^2$, whereas $m^2-n^2\leq3n^2$.
Thus the congruence has a nonzero difference.
Moreover, $bn^2-am^2>0$ and $tn-sm>0$, the latter because
$t^2-4s^2=c(b-4a)-3>0$. Therefore
\[
 4c\leq\bigl|bn^2-am^2+\lambda(tn-sm)\bigr|
 \leq bn^2+tn-1.
\]
If $n\leq\kappa\sqrt{c/b}$, then $b\geq9>\kappa$ gives
$n<\gamma$, and hence
$tn-1=n\sqrt{bc}+n/\gamma-1<n\sqrt{bc}$.
This implies $4c<bn^2+n\sqrt{bc}\leq(\kappa^2+\kappa)c=4c$,
a contradiction.
\end{proof}

\begin{lemma}\label{lem:higher-type-I-lower}
Every type-I irregular extension of a triple of degree
at least two satisfies
\[
 m>\kappa c^{1/12}>1.56c^{1/12}.
\]
\end{lemma}

\begin{proof}
If $c>b^{6/5}$, Lemma~\ref{lem:higher-type-I-first-lower} gives
$m>n>\kappa\sqrt{c/b}>\kappa c^{1/12}$.
Assume now that $c\leq b^{6/5}$.
Since $c>4ab$, we have $1\leq a<c^{1/6}/4$, so necessarily $c>4096$.
We first prove $2asm^3>c$.

Put $bn^2+\lambda tn-am^2-\lambda sm=kc$, with $k\in\mathbf Z$.
For the next term in the even recurrence, put
\[
 F_a(j)=\lambda+2c(\lambda aj^2+sj)
   +\frac23ac^2(j^3-j)(2s+\lambda aj).
\]
This integer-valued polynomial agrees with $z^{(a,c)}_{2j}$ at $j=0,1$, and
\[
 F_a(j+2)-(4ac+2)F_a(j+1)+F_a(j)
 =-\frac83a^2c^3j(j+1)(j+2)\bigl(2s+\lambda a(j+1)\bigr).
\]
Since $6\mid j(j+1)(j+2)$, induction yields
$z^{(a,c)}_{2j}\equiv F_a(j)\pmod{16c^3}$.
Use the analogous expansion for $(b,c)$ and put
\[
 A=a(m^3-m)(2s+\lambda am)-b(n^3-n)(2t+\lambda bn).
\]
As $3\mid A$, subtraction gives
$16c^3\mid2c^2(A/3-\lambda k)$, hence
\[
 A\equiv3\lambda k\pmod{24c}.
\]
Put $P=am+\lambda s$ and $Q=bn+\lambda t$.
Squaring $mP\equiv nQ\pmod c$ gives
$m^2P^2\equiv n^2Q^2\pmod c$. Since $s^2\equiv t^2\equiv1\pmod c$,
\begin{align*}
 \lambda A
 &=(m^2-1)(P^2-s^2)-(n^2-1)(Q^2-t^2)\\
 &\equiv Q^2+n^2-P^2-m^2\pmod c.
\end{align*}
Together with $\lambda A\equiv3k\pmod c$, this gives
\[
 P^2+m^2\equiv Q^2+n^2-3k\pmod c.
\]
Multiply by $n^2$ and use $n^2Q^2\equiv m^2P^2$ to obtain
\begin{equation}\label{eq:higher-type-I-key-cong}
 3kn^2\equiv R\pmod c,\qquad
 R=(m^2-n^2)(a^2m^2+2\lambda asm+1-n^2).
\end{equation}

Assume $2asm^3\leq c$. Since $s>\sqrt{ac}$,
$m^3<c^{1/2}/(2a^{3/2})$ and $m<c^{1/6}$.
In particular,
\[
 a^2m^4<\frac25c^{2/3},\qquad m^4<\frac25c^{2/3},\qquad
 m^2<\frac23c^{1/3}.
\]
Using $0<m^2-n^2\leq3m^2/4$ and $a,b,s,t<c$, we get
\begin{align*}
 |R|&<\frac34c+\frac35c^{2/3}+\frac12c^{1/3},\\
 3|k|n^2&<6m^4+6m^3<\frac{12}{5}c^{2/3}+3c^{1/2}.
\end{align*}
Thus
\[
 |R-3kn^2|<\frac34c+3c^{2/3}+3c^{1/2}+\frac12c^{1/3}
 <\frac{505}{512}c<c,
\]
where $c>4096$. The congruence is therefore the equality $R=3kn^2$.
Furthermore,
\[
 \frac{a^2m^2+1+n^2}{asm}
 \leq\frac{am}{s}+\frac1{asm}+\frac{m}{as}
 <\frac12c^{-1/4}+c^{-1/2}+c^{-1/3}
 <\frac9{64}<\frac12.
\]
It follows that $|a^2m^2+2\lambda asm+1-n^2|>3asm/2$.
Together with $m^2-n^2>2n$, the equality and the defining formula for $k$ give
\[
 c^{1/2}<as<\frac{asm}{n}<|k|
 <2m^2+2m\leq3m^2<3c^{1/3}.
\]
This would imply $c<729$, contrary to $c>4096$.
Hence $2asm^3>c$.

Finally $ac>4096$ gives $s^2=ac+1<(17/16)ac$.
Together with $a<c^{1/6}/4$, this yields
\[
 c<2asm^3<\frac{\sqrt{17}}{16}c^{3/4}m^3,
 \qquad
 m^3>\frac{16}{\sqrt{17}}c^{1/4}>\kappa^3c^{1/4}.
\]
The last comparison uses $\kappa^3=(5\sqrt{17}-13)/2<16/\sqrt{17}$,
equivalently $53<13\sqrt{17}$. Thus $m>\kappa c^{1/12}$ also in this range.
\end{proof}

\subsection{Type II}

We have $m,n\geq4$ and $n\leq m\leq2n+1$ by
Lemmas~\ref{lem:constant-index-bounds} and \ref{lem:mn-comparison}.
Applying the odd congruence in Lemma~\ref{lem:quadratic-expansions}
with $x_1=y_2=r$ and
$(z_1,z_2)=(\lambda t,\lambda s)$ gives
\begin{equation}\label{eq:higher-type-II-first}
 bn(n+1)-am(m+1)\equiv\lambda rst\Delta\pmod c,
 \qquad\lambda\in\{\pm1\}.
\end{equation}
Put
\[
 N=2n(n+1)-\lambda\Delta,\qquad
 M=2m(m+1)+\lambda\Delta,\qquad L=\lambda\Delta.
\]
Since $2rst\equiv a+b-e\pmod c$, we obtain
\begin{equation}\label{eq:higher-predecessor-cong}
 c\mid H,\qquad H=bN-aM+eL.
\end{equation}
The index bounds give
\[
 0<M<3m^2,\qquad 0<N\leq\frac52m^2,\qquad \Delta\leq\frac{m^2}{4}.
\]
Indeed,
$N\leq2n^2+n+m\leq2m^2+2m$ and $m\geq4$.

\begin{lemma}\label{lem:higher-bridge}
Suppose that $c>b^2$ and $m\leq c^{1/12}$. Then
$e>b$, $e=d_+(a,b,f)$, $H=0$, $\lambda=-1$,
$\Delta\geq1$, $M<4N$ and $N<2m^2$.
Moreover, if $X=c/(ab^2)$, then
\begin{equation}\label{eq:higher-ebX}
 X>16af,\qquad \frac{2X}{9}<\frac eb<\frac X4,
\end{equation}
and
\begin{equation}\label{eq:higher-XDelta}
 e\Delta=bN-aM,\qquad 2N<X\Delta<\frac92N.
\end{equation}
\end{lemma}

\begin{proof}
We first show that $e>b$.
If $e<b$, the inverse-root identity gives $b=d_+(a,e,f)$, so
$b>4aef+a+e+f>4ae+4$.
On the other hand, $c=d_+(a,b,e)<4b(ae+1)$ and $c>b^2$ give
$b<4ae+4$, a contradiction.
Thus $e>b$ and $e=d_+(a,b,f)$.
Consequently $c>4abe>16a^2b^2f$, giving $X>16af$.
Also $4abe<c<4e(ab+1)$, so
$Xab/[4(ab+1)]<e/b<X/4$.
Since $ab\geq8$, this proves \eqref{eq:higher-ebX}.

By \cite[Theorem~1.4\textup{(4)}]{CipuFujitaMiyazaki}, an irregular
extension satisfies $c<200b^4$.
Using $a<b<c^{1/2}$ and $e<c/(4ab)$, we get
\[
 |H|<6bm^2+\frac{cm}{4b}
 <6c^{2/3}+c^{5/6}<c,
\]
since $200^{1/4}<4$ and $c>b^2>4000^2>12^6$.
Hence $H=0$ by \eqref{eq:higher-predecessor-cong}.
If $\Delta=0$, this equality becomes $(b-a)N=0$, which is impossible.
For $\lambda=1$, substitution of $m=n+\Delta$, with
$\Delta\leq n+1$, gives $8N-M\geq8n^2-5n-13>0$.
Thus $bN=aM-e\Delta<aM<bN$, a contradiction.
Therefore $\lambda=-1$.
The same substitution gives $4N-M\geq n+1>0$, and hence
\[
 \frac12bN<e\Delta=bN-aM<bN.
\]
Combining this with \eqref{eq:higher-ebX} proves
\eqref{eq:higher-XDelta}.
Finally, $n\leq m-1$ gives
$N=2n^2+n+m\leq2m^2-2m+1<2m^2$.
\end{proof}

\begin{proposition}\label{prop:higher-type-II-lower}
Every higher-degree type-II irregular extension satisfies
\[
 m>\frac1{\sqrt3}\,c^{1/16}.
\]
\end{proposition}

\begin{proof}
Suppose that $m\leq c^{1/16}/\sqrt3$.
First assume $c>b^2$. Lemma~\ref{lem:higher-bridge} applies.
Put $w=\sqrt{af+1}$, $R=(a+f)N-aM$ and $Q=2wN$.
The regular-root product identity and $bN\equiv aM\pmod e$ give
\begin{equation}\label{eq:higher-second-cong}
 R^2\equiv Q^2\pmod e.
\end{equation}
Since $af\geq3$ and $X>16af$,
\[
 a+f\leq af+1\leq\frac{4af}{3}<\frac{X}{12},\qquad
 2w\leq\frac{4af}{3}<\frac{X}{12}.
\]
Thus $|R|<XN/3$ and $Q<XN/12$.
If $R^2\ne Q^2$, then
\[
 \frac{2Xb}{9}<e\leq|R^2-Q^2|
 <\frac{25}{144}X^2N^2,
\]
so $b<\frac{25}{32}XN^2<XN^2$.

If $R^2=Q^2$, put $F_\pm=a+f\pm2w$.
The pair condition gives $|a-f|\geq2$, hence
$F_-F_+=(a-f)^2-4\geq0$ and $F_-\geq0$.
For some $\mu\in\{\pm1\}$ we have
$aM=F_\mu N$ and $e\Delta=(b-F_\mu)N$.
Let $g=d_-(a,b,f)$. The identities
\begin{align*}
 e+g&=(4af+2)b+2(a+f),\\
 eg&=(b-F_\mu)(b-F_{-\mu})
\end{align*}
yield
\[
 b\bigl(N^2-(4af+2)N\Delta+\Delta^2\bigr)
 =F_\mu N^2+2(a+f)N\Delta+F_{-\mu}\Delta^2.
\]
The right side is positive, so the integer factor on the left is at least
one. Since $\Delta<N$ and $F_\pm<X/6$, this gives
$b<XN^2/2<XN^2$.
In either case, $a<X/16$, $X<\frac92N$ and $N<2m^2$ imply
\[
 c=Xab^2<\frac1{16}X^4N^4
 <\frac{6561}{256}N^8<3^8m^{16}\leq c,
\]
a contradiction.

Now let $c\leq b^2$. Then $e<b/(4a)\leq b/4$.
Substituting $m=n+\Delta$ and $0\leq\Delta\leq n+1$ gives, for either sign,
$8N-M-2\Delta\geq8n^2-7n-15>0$.
Therefore
\[
 H\geq bN-aM-e\Delta>b\left(N-\frac M8-\frac\Delta4\right)>0,
\]
so $H\geq c$ by \eqref{eq:higher-predecessor-cong}.
If $c\geq b^{5/4}$, then
\[
 c\leq H<bN+e\Delta\leq\frac{41}{16}bm^2
 \leq\frac{41}{48}c^{37/40}<c,
\]
a contradiction.

It remains to consider $c<b^{5/4}$.
Put $S=a+e$, $u=\sqrt{ae+1}$ and $P=aM-eL$.
We have $ae<c^{1/5}/4$, so
$S\leq u^2\leq4ae/3<c^{1/5}/3$.
Since $P\equiv bN\pmod c$ and
$cf=(b-a-e)^2-4(ae+1)$,
\begin{equation}\label{eq:higher-small-square-c}
 (P-SN)^2\equiv4u^2N^2\pmod c.
\end{equation}
Both sides are less than $c$, because
\[
 |P-SN|<6Sm^2<\frac23c^{13/40},\qquad
 4u^2N^2<\frac{25}{3}c^{1/5}m^4
 \leq\frac{25}{27}c^{9/20}.
\]
Thus \eqref{eq:higher-small-square-c} is an equality. Put
$E_\pm=S\pm2u$. As above, $E_-\geq0$, so for some
$\mu\in\{\pm1\}$ we have
\[
 P=E_\mu N\geq0.
\]
Moreover, $E_+\leq2S<3b/4$, and hence
\[
 J=\frac{bN-P}{c}=\frac{(b-E_\mu)N}{c}\in\mathbf Z,
 \qquad 0<J<\frac N4.
\]
On the other hand, with $v=\sqrt{be+1}$, the regular-root formula
$c=a+b+e+2abe+2ruv$ gives
\begin{equation}\label{eq:higher-small-square-b}
 (-P-JS)^2\equiv4J^2u^2\pmod b.
\end{equation}
Again this is an equality. Indeed, since $b>c^{4/5}$, the bounds
\[
 P+JS<\frac94SN<\frac34c^{13/40},\qquad
 4J^2u^2<\frac14u^2N^2<\frac1{12}c^{9/20}
\]
show that both squared sides are smaller than $b$.
Consequently, for some $\nu\in\{\pm1\}$,
\[
 -P=JE_\nu\geq0.
\]
It follows that $P=0$. Hence $E_-=0$, $|a-e|=2$, and $aM=eL$
forces $\lambda=1$ and $\Delta>0$. But then
\[
 \frac ea=\frac M\Delta>2m\geq8,\qquad
 \frac ea\leq3,
\]
a contradiction.
\end{proof}

\begin{proposition}\label{prop:higher-index-lower}
Every irregular extension of a triple of degree at least two satisfies
\begin{equation}\label{eq:higher-index-lower}
 m>\frac1{\sqrt3}\,c^{1/16}.
\end{equation}
\end{proposition}

\begin{proof}
For type I, Lemma~\ref{lem:higher-type-I-lower} gives
$m>\kappa c^{1/12}>c^{1/16}/\sqrt3$.
For type II, use Proposition~\ref{prop:higher-type-II-lower}.
\end{proof}

\section{A common three-logarithm estimate}\label{sec:linear-form}

We record the form of Matveev's theorem used below.

For an algebraic number $\eta$ of degree $D$ with primitive minimal
polynomial $a_0\prod_{j=1}^D(X-\eta^{(j)})\in\mathbf Z[X]$, $a_0>0$,
its absolute logarithmic height is
\[
 h(\eta)=\frac1D\left(\log a_0+
 \sum_{j=1}^D\log\max\{1,|\eta^{(j)}|\}\right).
\]

\begin{theorem}[Matveev {\cite[Theorem~2.2]{Matveev}}]\label{thm:matveev}
Let $K\subset\mathbf R$ have degree $d_K$, and let
\[
 \Lambda=b_1\log\eta_1+b_2\log\eta_2+b_3\log\eta_3\ne0,
\]
where
$\eta_i>0$ are multiplicatively independent algebraic numbers in $K$,
$b_i\in\mathbf Z$, and $b_3\ne0$. Choose
$A_i\geq\max\{d_Kh(\eta_i),|\log\eta_i|,0.16\}$ for $1\leq i\leq3$, and put
\[
 B=\max\left\{1,\frac{|b_1|A_1}{A_3},\frac{|b_2|A_2}{A_3},|b_3|\right\},
 \qquad\mathcal A=A_1A_2A_3.
\]
Then
\begin{equation}\label{eq:matveev-original}
 \log|\Lambda|>-C(3)C_0W_0d_K^2\mathcal A,
\end{equation}
where
\begin{align*}
 C(3)&=\frac{16}{3!}\mathrm e^3\cdot9\cdot5\cdot16^4
 \cdot\frac{3\mathrm e}{2}<6.441\cdot10^8,\\
 C_0&=\log\!\left(\mathrm e^{20.2}3^{5.5}d_K^2\log(\mathrm e d_K)\right),\\
 W_0&=\log\!\left(1.5\mathrm e Bd_K\log(\mathrm e d_K)\right).
\end{align*}
\end{theorem}

Put
\begin{equation}\label{eq:linear-prefactor}
 \delta=\frac{\sqrt b\,(x_1\sqrt c+z_1\sqrt a)}
 {\sqrt a\,(y_2\sqrt c+z_2\sqrt b)}
 =\frac ba\frac{x_1\sqrt{ac}+az_1}{y_2\sqrt{bc}+bz_2}.
\end{equation}
The common Pell representations give the following upper estimate.

\begin{lemma}[{\cite[Lemma~5]{DujellaFiniteness}}]
\label{lem:linear-form-upper}
If \(m',n'>0\), then
\begin{equation}\label{eq:linear-form-upper}
 0<\Lambda:=
 m'\log\beta-n'\log\gamma+\log\delta
 <\frac83ac\,\beta^{-2m'}.
\end{equation}
\end{lemma}

\begin{proposition}\label{prop:three-log-bound}
Every irregular extension $d>d_+(a,b,c)$ satisfies
\begin{equation}\label{eq:three-log-bound}
 m'<2+4.436\cdot10^{13}(\log c)^2\log\!\bigl(39(m'+1)\bigr).
\end{equation}
\end{proposition}

\begin{proof}
Set $K=\mathbf Q(\sqrt{ab},\sqrt{ac})$. None of $ab,ac,bc$ is a square,
since each lies strictly between consecutive integer squares.
Thus $ab$ and $ac$ represent distinct nontrivial square classes, so
$[K:\mathbf Q]=4$ and $\beta,\gamma,\delta\in K$.
The numerator and denominator in \eqref{eq:linear-prefactor} are positive,
since $cx_1^2-az_1^2=c-a>0$ and $cy_2^2-bz_2^2=c-b>0$.
The units $\beta,\gamma$ are independent, since an automorphism of $K$ inverts
one and fixes the other. Moreover,
\[
 N_{K/\mathbf Q}(\beta)=N_{K/\mathbf Q}(\gamma)=1,\qquad
 N_{K/\mathbf Q}(\delta)=\left(\frac{b(c-a)}{a(c-b)}\right)^2>1.
\]
Taking norms in $\beta^u\gamma^v\delta^w=1$ forces $w=0$, and hence
$u=v=0$. This proves the independence required by Matveev's theorem.

The normalized initial coordinates are $(x_1,y_2)=(1,1)$ or $(r,r)$,
and $(|z_1|,|z_2|)=(1,1)$ or $(t,s)$. All are less than $c$.
Every conjugate of the numerator and denominator in the second quotient
of \eqref{eq:linear-prefactor} has absolute value less than $2c^2$.
Since $c>b>4000$ by Lemma~\ref{lem:higher-b-absolute}, the height inequalities give
\[
 h(\beta)=\tfrac12\log\beta,\qquad
 h(\gamma)=\tfrac12\log\gamma,\qquad
 h(\delta)<5\log c+2\log2<6\log c.
\]
In particular $|\log\delta|\leq4h(\delta)<36\log c$.
Apply Matveev's theorem with
\[
 (\eta_1,\eta_2,\eta_3)=(\beta,\gamma,\delta),\qquad
 (b_1,b_2,b_3)=(m',-n',1).
\]
We may take
\[
 A_1=2\log\beta,\qquad A_2=2\log\gamma,\qquad A_3=36\log c.
\]
Since $\beta<\gamma<3c<c^2$ and $n'\leq m'+1$, we have
$B\leq\max\{1,(m'+1)/9\}\leq m'+1$. For $d_K=4$,
\[
 C_0<29.885,\qquad C(3)C_0d_K^2<3.080\cdot10^{11},\qquad
 W_0<\log\!\bigl(39(m'+1)\bigr).
\]
As $A_1A_2A_3=144\log\beta\log\gamma\log c$, Matveev's theorem and
Lemma~\ref{lem:linear-form-upper} yield
\begin{align*}
 -4.436\cdot10^{13}\log\beta\log\gamma\log c\,
 \log\!\bigl(39(m'+1)\bigr)
 &<\log\Lambda\\
 &<\log(8ac/3)-2m'\log\beta.
\end{align*}
Finally, $\log(8ac/3)<4\log\beta$ and $\log\gamma<2\log c$
give \eqref{eq:three-log-bound}.
\end{proof}

We now combine \eqref{eq:three-log-bound} with the index lower bounds
for the three degree classes.

\begin{proposition}[Euler triples]\label{prop:euler-upper}
If an Euler triple admits an irregular extension, then
\[
 c<10^{98}.
\]
\end{proposition}

\begin{proof}
Suppose that $c\geq10^{98}$, and use the reindexed representations in
Lemma~\ref{lem:euler-type-I}, so that $m'=2m$.
Proposition~\ref{prop:euler-m5} gives
\[
 m'^5>c\log c,\qquad m'>1.176\cdot10^{20}.
\]
In particular, $\log c<5\log m'-\log(98\log10)$.
Combining this with \eqref{eq:three-log-bound}, we obtain
\[
 \frac{m'-2}
 {(5\log m'-\log(98\log10))^2\log\!\bigl(39(m'+1)\bigr)}
 <4.436\cdot10^{13}.
\]
The left-hand side is increasing for $m'\geq10^3$ and exceeds
$4.53\cdot10^{13}$ at $m'=1.15\cdot10^{20}$.
Thus $m'<1.15\cdot10^{20}$, a contradiction.
\end{proof}

\begin{proposition}[Degree-one triples]\label{prop:degree-one-upper}
If a triple of degree one admits an irregular extension, then
\[
 c<10^{162}.
\]
\end{proposition}

\begin{proof}
By Proposition~\ref{prop:degree-one-m-lower}, $m'\geq2m>2c^{1/8}$.
Hence $\log c<8\log(m'/2)$, and \eqref{eq:three-log-bound} gives
\[
 \frac{m'-2}{(\log(m'/2))^2\log\!\bigl(39(m'+1)\bigr)}
 <2.83904\cdot10^{15}.
\]
For $m'\geq10^3$, the left-hand side is increasing and exceeds
$2.90\cdot10^{15}$ at $m'=3.2\cdot10^{20}$. Thus $m'<3.2\cdot10^{20}$,
which also holds when $m'<10^3$. Consequently,
\[
 c<(m'/2)^8<4.3\cdot10^{161}<10^{162}.
\]
\end{proof}

\begin{proposition}[Higher-degree triples]\label{prop:higher-upper}
If a triple of degree at least two admits an irregular extension, then
\[
 c<10^{338}.
\]
\end{proposition}

\begin{proof}
Proposition~\ref{prop:higher-index-lower} gives
$m'\geq2m>(2/\sqrt3)c^{1/16}$.
Combining $\log c<16\log(\sqrt3\,m'/2)$ with
\eqref{eq:three-log-bound}, we obtain
\[
 \frac{m'-2}{(\log(\sqrt3\,m'/2))^2\log\!\bigl(39(m'+1)\bigr)}
 <1.135616\cdot10^{16}.
\]
For $m'\geq10^3$, the left-hand side is increasing and exceeds
$1.21\cdot10^{16}$ at $m'=1.5\cdot10^{21}$. Hence $m'<1.5\cdot10^{21}$,
also trivially true for $m'<10^3$, and
\[
 c<(\sqrt3\,m'/2)^{16}<6.6\cdot10^{337}<10^{338}.
\]
\end{proof}

\begin{proof}[Proof of Theorem~\ref{thm:main}]
Let $\{a,b,c,d\}$ be an irregular Diophantine quadruple with $a<b<c<d$.
By Lemma~\ref{lem:irregular-above-regular}, we have $d>d_+(a,b,c)$.
Propositions~\ref{prop:euler-upper}, \ref{prop:degree-one-upper},
and~\ref{prop:higher-upper} therefore give $c<10^{338}$.

The function $(u-2)/\log(39(u+1))$ is increasing for $u>2$.
Its value at $u=1.45\cdot10^{21}$ exceeds
$4.436\cdot10^{13}(338\log10)^2$, so \eqref{eq:three-log-bound} gives
$m'<1.45\cdot10^{21}$.
From \eqref{eq:pell-sequence-ac} and \eqref{eq:common-coordinates},
\[
 x=\frac12\left(
 \left(x_1+z_1\sqrt{a/c}\right)\beta^{m'}
 +\left(x_1-z_1\sqrt{a/c}\right)\beta^{-m'}\right)
 <x_1\beta^{m'}.
\]
Both coefficients are positive because $cx_1^2-az_1^2=c-a>0$.
Since $x_1\in\{1,r\}$, we have $x_1^2/a\leq b+1\leq c$.
Also $\beta<2c$, and hence
$d=(x^2-1)/a<c(2c)^{2m'}$. Consequently,
\begin{align*}
 \log_{10}d
 &<\log_{10}c+2m'\log_{10}(2c)\\
 &<338+2.9\cdot10^{21}(338+\log_{10}2)<10^{24}.
\end{align*}
This proves the asserted bounds. Since all four elements are positive
integers bounded by $10^{10^{24}}$, the finiteness assertion follows.
\end{proof}

\section*{Note on this version}

The mathematical ideas in this paper are due to the authors.
AI tools developed by OpenAI(ChatGPT)  were used to assist with checking the
arguments and computing the numerical constants.

This version was prepared under time constraints. The authors take
responsibility for the correctness of the mathematical arguments in
this version. A subsequent version will give a fuller account of the
literature and further refine the notation, exposition, and logical
organization of the proofs.


\begin{thebibliography}{99}

\bibitem{AdedjiFilipinTogbePell}
K.~N. Ad\'edji, A.~Filipin, and A.~Togb\'e,
\emph{On the family of Diophantine pairs \(\{P_{2k},2P_{2k+2}\}\)},
Fibonacci Quart. \textbf{60} (2022), 25--39.

\bibitem{AdedjiHePinterTogbe}
K.~N. Ad\'edji, B.~He, \'{A}.~Pint\'er, and A.~Togb\'e,
\emph{On the Diophantine pair \(\{a,3a\}\)},
J. Number Theory \textbf{227} (2021), 330--351.

\bibitem{AdzagaFilipinJurasic}
N.~Ad\v{z}aga, A.~Filipin, and A.~Jurasi\'c,
\emph{The extensibility of the Diophantine triple \(\{2,b,c\}\)},
An. \c{S}tiin\c{t}. Univ. Ovidius Constan\c{t}a Ser. Mat.
\textbf{29} (2021), no.~2, 5--24.

\bibitem{ArkinHoggattStraus}
J.~Arkin, V.~E. Hoggatt, Jr., and E.~G. Straus,
\emph{On Euler's solution of a problem of Diophantus},
Fibonacci Quart. \textbf{17} (1979), 333--339.

\bibitem{AssafGueron}
E.~Assaf and S.~Gueron,
\emph{Characterization of regular Diophantine quadruples},
Elem. Math. \textbf{56} (2001), 71--81.

\bibitem{BakerDavenport}
A.~Baker and H.~Davenport,
\emph{The equations \(3x^2-2=y^2\) and \(8x^2-7=z^2\)},
Quart. J. Math. Oxford Ser. (2) \textbf{20} (1969), 129--137.

\bibitem{BugeaudDujellaMignotte}
Y.~Bugeaud, A.~Dujella, and M.~Mignotte,
\emph{On the family of Diophantine triples \(\{k-1,k+1,16k^3-4k\}\)},
Glasgow Math. J. \textbf{49} (2007), 333--344.

\bibitem{CipuDujellaFujita}
M.~Cipu, A.~Dujella, and Y.~Fujita,
\emph{Extensions of a Diophantine triple by adjoining smaller elements},
Mediterr. J. Math. \textbf{19} (2022), Art. 187.

\bibitem{CipuFilipinFujitaBoundsII}
M.~Cipu, A.~Filipin, and Y.~Fujita,
\emph{Bounds for Diophantine quintuples II},
Publ. Math. Debrecen \textbf{88} (2016), 59--78.

\bibitem{CipuFilipinFujitaFamily}
M.~Cipu, A.~Filipin, and Y.~Fujita,
\emph{An infinite two-parameter family of Diophantine triples},
Bull. Malays. Math. Sci. Soc. \textbf{43} (2020), 481--498.

\bibitem{CipuFilipinFujitaPairs}
M.~Cipu, A.~Filipin, and Y.~Fujita,
\emph{Diophantine pairs that induce certain Diophantine triples},
J. Number Theory \textbf{210} (2020), 433--475.

\bibitem{CipuFujitaMignotte}
M.~Cipu, Y.~Fujita, and M.~Mignotte,
\emph{Two-parameter families of uniquely extendable Diophantine triples},
Sci. China Math. \textbf{61} (2018), 421--438.

\bibitem{CipuFujitaMignotteThree}
M.~Cipu, Y.~Fujita, and M.~Mignotte,
\emph{Diophantine triples with three parameters},
Res. Number Theory \textbf{10} (2024), Art.~44.

\bibitem{CipuFujitaMiyazaki}
M.~Cipu, Y.~Fujita, and T.~Miyazaki,
\emph{On the number of extensions of a Diophantine triple},
Int. J. Number Theory \textbf{14} (2018), 899--917.

\bibitem{DicksonHistory}
L.~E. Dickson,
\emph{History of the Theory of Numbers, Vol.~II},
Chelsea, New York, 1966, 514--519.

\bibitem{Diophantus}
Diophantus of Alexandria,
\emph{Arithmetics and the Book of Polygonal Numbers},
I.~G. Bashmakova (ed.), Nauka, Moscow, 1974, 103--104, 232.

\bibitem{DujellaParametric}
A.~Dujella,
\emph{The problem of the extension of a parametric family of Diophantine triples},
Publ. Math. Debrecen \textbf{51} (1997), 311--322.

\bibitem{DujellaHoggattBergum}
A.~Dujella,
\emph{A proof of the Hoggatt--Bergum conjecture},
Proc. Amer. Math. Soc. \textbf{127} (1999), 1999--2005.

\bibitem{DujellaAbsolute}
A.~Dujella,
\emph{An absolute bound for the size of Diophantine \(m\)-tuples},
J. Number Theory \textbf{89} (2001), 126--150.

\bibitem{DujellaFiniteness}
A.~Dujella,
\emph{There are only finitely many Diophantine quintuples},
J. Reine Angew. Math. \textbf{566} (2004), 183--214.

\bibitem{DujellaPetho}
A.~Dujella and A.~Peth\H{o},
\emph{A generalization of a theorem of Baker and Davenport},
Quart. J. Math. Oxford Ser. (2) \textbf{49} (1998), 291--306.

\bibitem{FilipinFujitaTogbeExamples}
A.~Filipin, Y.~Fujita, and A.~Togb\'e,
\emph{The extendibility of Diophantine pairs II: Examples},
J. Number Theory \textbf{145} (2014), 604--631.

\bibitem{FujitaPair}
Y.~Fujita,
\emph{The extensibility of Diophantine pairs \(\{k-1,k+1\}\)},
J. Number Theory \textbf{128} (2008), 322--353.

\bibitem{FujitaQuintupleRegular}
Y.~Fujita,
\emph{Any Diophantine quintuple contains a regular Diophantine quadruple},
J. Number Theory \textbf{129} (2009), 1678--1697.

\bibitem{FujitaFixed}
Y.~Fujita,
\emph{The number of irregular Diophantine quadruples for a fixed Diophantine pair or triple},
in \emph{Lie Groups, Number Theory, and Vertex Algebras},
Contemp. Math. \textbf{768}, Amer. Math. Soc., Providence, RI, 2021, 105--118.

\bibitem{FujitaMiyazaki}
Y.~Fujita and T.~Miyazaki,
\emph{The regularity of Diophantine quadruples},
Trans. Amer. Math. Soc. \textbf{370} (2018), 3803--3831.

\bibitem{GibbsRegular}
P.~E. Gibbs,
\emph{A generalised Stern--Brocot tree from regular Diophantine quadruples},
arXiv:math/9903035 (1999).

\bibitem{HePinterTogbeYang}
B.~He, \'{A}.~Pint\'er, A.~Togb\'e, and S.~Yang,
\emph{Another generalization of a theorem of Baker and Davenport},
J. Number Theory \textbf{182} (2018), 325--343.

\bibitem{HeEtAlPair}
B.~He, K.~Pu, R.~Shen, and A.~Togb\'e,
\emph{A note on the regularity of the Diophantine pair \(\{k,4k\pm4\}\)},
J. Th\'eor. Nombres Bordeaux \textbf{30} (2018), 879--892.

\bibitem{HeTogbeFamilyI}
B.~He and A.~Togb\'e,
\emph{On a family of Diophantine triples
\(\{k,A^2k+2A,(A+1)^2k+2(A+1)\}\) with two parameters},
Acta Math. Hungar. \textbf{124} (2009), 99--113.

\bibitem{HeTogbeK4K9K}
B.~He and A.~Togb\'e,
\emph{On the family of Diophantine triples
\(\{k+1,4k,9k+3\}\)},
Period. Math. Hungar. \textbf{58} (2009), 59--70.

\bibitem{HeTogbeFamilyII}
B.~He and A.~Togb\'e,
\emph{On a family of Diophantine triples
\(\{k,A^2k+2A,(A+1)^2k+2(A+1)\}\) with two parameters II},
Period. Math. Hungar. \textbf{64} (2012), 1--10.

\bibitem{HeTogbeZiegler}
B.~He, A.~Togb\'e, and V.~Ziegler,
\emph{There is no Diophantine quintuple},
Trans. Amer. Math. Soc. \textbf{371} (2019), 6665--6709.

\bibitem{Matveev}
E.~M. Matveev,
\emph{An explicit lower bound for a homogeneous rational linear form in the logarithms of algebraic numbers. II},
Izv. Math. \textbf{64} (2000), 1217--1269.

\bibitem{ParkLeeFibonacci}
J.~Park and J.~B. Lee,
\emph{Some family of Diophantine pairs with Fibonacci numbers},
Indian J. Pure Appl. Math. \textbf{50} (2019), 367--384.

\end{thebibliography}
\end{document}